\documentclass{amsart}
\usepackage{amssymb}
\usepackage{amsmath}
\numberwithin{equation}{section}
\usepackage[dvipsnames]{xcolor}
\usepackage[hidelinks,colorlinks=true,citecolor=ForestGreen,linkcolor=RubineRed]{hyperref}
\usepackage{amsthm}
\usepackage{appendix}
\usepackage{bookmark}
\usepackage{setspace}
\usepackage{mathrsfs}
\usepackage{mathtools}
\usepackage{tikz-cd}
\usepackage{manfnt}
\usepackage{fullpage}

\newtheorem{thm}{Theorem}
\numberwithin{thm}{section}
\newtheorem{prop}[thm]{Proposition}
\newtheorem{lem}[thm]{Lemma}
\newtheorem{cor}[thm]{Corollary}
\theoremstyle{remark}
\newtheorem{rem}[thm]{Remark}
\theoremstyle{definition}

\newtheorem{eg}[thm]{Example}

\newcommand{\1}{\mathbf{1}}

\newcommand{\C}{\mathbf{C}}
\newcommand{\cali}[1]{\mathcal{#1}}
\newcommand{\chx}[1]{\langle #1\rangle}
\newcommand{\cl}{\mathrm{Cl}}

\newcommand{\comment}[1]{}

\newcommand{\F}{\mathbf{F}}

\newcommand{\gal}{\mathrm{Gal}}

\newcommand{\got}[1]{\mathfrak{#1}}

\newcommand{\plim}{\varprojlim}

\newcommand{\Q}{\mathbf{Q}}

\newcommand{\R}{\mathbf{R}}

\newcommand{\tr}{\mathrm{Tr}}

\newcommand{\Z}{\mathbf{Z}}

\newcommand{\nm}{\mathrm{Nm}}
\newcommand{\ord}{\mathrm{ord}}

\title{On a product formula of bivariate $p$-adic Gamma values}

\author{Luochen Zhao}

\date{Nov 8, 2023}

\subjclass[2020]{11S40 (primary); 11S80, 11Y35 (secondary).}

\keywords{Ferrero--Greenberg type derivative formulas, multivariate $p$-adic Gamma functions, genus $L$-functions, Davenport--Hasse relations}

\address{Einstein Institute of Mathematics, Edmond J. Safra Campus, The Hebrew University of Jerusalem, Givat Ram, Jerusalem, 9190401, Israel}
\email{zhao.luochen@mail.huji.ac.il}
\begin{document}
	\maketitle
	
	\begin{abstract}
		We establish a derivative formula of $p$-adic Shintani $L$-functions, thus those of totally real $p$-adic Hecke $L$-functions with trivial moduli. As an application, we present a product formula of bivariate $p$-adic Gamma values by specializing to genus $L$-functions, which takes a particularly simple form when the field is $\Q(\sqrt{3})$. In the appendix we explain how the product formula is supposed to encode a higher analytic analogue of the Davenport-Hasse relation.
	\end{abstract}
	
	\setcounter{tocdepth}{1}
	\tableofcontents
	
	\section{Introduction}
	
	Let $p$ be a prime number and $F$ be a totally real field with ring of integers $\cali{O}$. Let $\chi$ be a finite Hecke character over $F$, so there exists an integral ideal $\cali{I}$ for which $\chi$ can be regarded as a character on the narrow ray class group of modulus $\cali{I}$: $\chi: \cl_+(\cali{I}) \to \bar{\Q}^\times$.
	
	We call the ideal $\cali{I}$ Cassou-Nogu\`es \cite{CN79}, if $\cali{O}/\cali{I}$ is a cyclic group. In \cite[\S5]{Zh23}, under the premise $\cali{I}\ne \cali{O}$ together with some mild assumptions, we obtained a derivative formula of the $p$-adic Hecke $L$-function of $\chi$ in the spirit of Ferrero--Greenberg \cite{FG78}. In the present paper, we shall complement the aforementioned result by providing a derivative formula when $\cali{I}=\cali{O}$. To state it, we fix some notation first: Let $k$ be the degree of $F/\Q$. Denote by $\omega$ the Teichm\"uller character, $\omega_F$ the character $\omega\circ\nm_{F/\Q}$ on $\cl_+(p)$, and $\chx{\cdot}$ the character $x/\omega(x)$ on $\Z_p^\times$. Let $L_{F,p}(s,\chi\omega_F)$ be the $p$-adic meromorphic function on $\{s\in \C_p:|s|\le 1\}$ determined by the interpolation property
	\begin{align*}
		L_{F,p}(1-m,\chi\omega_F) = \prod_{\got{p}|p}(1-\chi\omega_F^{1-m}(\got{p})\nm(\got{p})^{m-1})L_F(1-m,\chi\omega_F^{1-m})
	\end{align*}
	for $m\in \Z_{\ge 1}$, where for a Hecke character $\psi$ of $F$, $L_F(s,\psi)$ denotes the complex $L$-function $\sum_{0\ne \got{a}\subseteq \cali{O}}\psi(\got{a})\nm\got{a}^{-s}$. Choose once and for all a set of representatives $\{\got{a}_i\}_{1\le i\le h}$ of $\cl_+(F)$. If $S\subseteq F$, denote by $S_+$ the subset of totally positive elements of $S$. As in \cite{Sh76}, we fix throughout a Shintani cone decomposition
	\begin{align}\label{eq:shintani}
		(F\otimes \R)_+/\cali{O}^\times_+ = \bigsqcup_V \overline{C}(V),
	\end{align}
	where each $V=\{v_1,v_2,\cdots,v_k\}\subset F_+$ spans a $k$-dimensional cone $C(V) = \sum_{1\le i\le k} \R_{>0}v_i$ and $\overline{C}(V)$ is the upper closure of $C(V)$ after Yamamoto \cite{Ya10}. Write $P(V)$ for the fundamental parallelotope of $\overline{C}(V)$, so we have $\overline{C}(V) = P(V) + \sum_{1\le i\le k}\Z_{\ge 0}v_i$. Finally, if $U=\{u_1,\cdots,u_k\}\subset \cali{O}_p$, define the multivariate $p$-adic Gamma function as
	\begin{align*}
		\Gamma_{p,U}(y_1,\cdots,y_k)=\Gamma_{p,U}(y_1u_1+\cdots+y_ku_k) = \lim_{\substack{n_1,\cdots,n_k>0\\ n_1\to y_1,\cdots,n_k\to y_k}} \prod_{\substack{1\le l_1<n_1,\cdots,1\le l_k<n_k\\ \gcd(p,l_1u_1+\cdots +l_ku_k)=1}}\chx{\nm(l_1u_1+\cdots +l_ku_k)}.
	\end{align*}
	
	Our first result can be stated as follows:
	\begin{thm}\label{thm.A}
		Let $\chi$ be a Hecke character of the narrow class group $\cl_+(F)$ and $\cali{N}$ be an auxiliary integral ideal of $F$. Suppose the following are valid:
		\begin{enumerate}
			\item[(A1)] $p$ is inert in $F$.
			\item[(A2)] $\cali{N}$ is Cassou-Nogu\`es, i.e., there is a group isomorphism $\cali{O}/\cali{N}\simeq \Z/N$ with $N = \#(\cali{O}/\cali{N})$, $\cali{N}\ne \cali{O}$ and $\cali{N}$ is prime to $p$.
			\item[(A3)] Each ideal representative $\got{a}_i$ is integral and prime to $\cali{N}p$.
			\item[(A4)] Any $V$ in the Shintani cone decomposition \eqref{eq:shintani} is contained in $\cali{O}$, and is a basis of $\cali{O}_p=\plim_n\cali{O}/p^n$. Furthermore, for all $v\in V$, $v$ is prime to $\cali{N}$.
		\end{enumerate}
		Then, we have the derivative formula
		\begin{align*}
			(1-\chi(\cali{N})N)L'_{F,p}(0,\chi\omega_F) = (-1)^k\sum_{1\le i\le h}\sum_V\sum_{x\in \got{a}_i^{-1}\cap P(V)}\chi(\got{a}_i)\sum_{0\le d<N}\got{h}_\emptyset(x+d\cdot v)\log_p\Gamma_{p,V}\left(\frac{x+d\cdot v}{N}\right),
		\end{align*}
		where $\got{h}_\emptyset(a) = -1 + N\1_{a\equiv 0\bmod \cali{N}}$.
	\end{thm}
	
	Specializing this to the case of genus $L$-functions and exploiting their factorizations, we shall find a product formula of bivariate $p$-adic Gamma values; see Proposition \ref{prop.prod-formula}. When $F=\Q(\sqrt{3})$, the product formula takes a particularly simple form:
	\begin{cor}\label{cor:B}
		Suppose $p$ is inert in $\Q(\sqrt{3})$, i.e., $p\equiv 5,7\bmod 12$. Put $\varepsilon = 2+\sqrt{3}$ and $V = \{1,\varepsilon\}$. Let $\Gamma_p(x)$ be Morita's $p$-adic Gamma function. Then, for any odd integer $N>1$ prime to $p$, not divisible by $9$, and all of whose prime divisors split or ramify in $\Q(\sqrt{3})$, choose an integral ideal $\cali{N}$ with $\Z[\sqrt{3}]/\cali{N}=\Z/N$. Denote by $\chi_{-4}$ the quadratic character of $\Q(i)$. We have
		\begin{align}\label{eq:product-formula}
			\frac{\prod_{0\le a,b<N}\Gamma_{p,V}\left(\frac{\frac{1+\varepsilon}{2}+a+b\varepsilon}{N}\right)}{\prod_{a+b\varepsilon\equiv -(1+\varepsilon)/2 \bmod \cali{N}} \Gamma_{p,V}\left(\frac{\frac{1+\varepsilon}{2}+a+b\varepsilon}{N}\right)^N}
			=
			\begin{cases}
				\chx{\Gamma_p(1/4)}^{4(1-\chi_{-4}(N)N)/3}&\text{if }p\equiv 5\bmod 12;\\
				\chx{\Gamma_p(1/3)}^{2(1-\chi_{-4}(N)N)}&\text{if }p\equiv 7\bmod 12.
			\end{cases}
		\end{align}
	\end{cor}
	\begin{rem}\label{rem:Gross-Stark}
		Since equation \eqref{eq:product-formula} might be regarded as an extension of the $p$-adic Chowla--Selberg formula \cite[(4.12)]{GK79}, one is curious about its arithmetic meaning. Put $F=\Q(\sqrt{3})$. Denote by $H/F$ (resp.~$H_0/F$) the ray class field corresponding to $\cl_+(\cali{N})/(p)^{\Z}$ (resp.~$\cl_+(F)$). Write $\bar{\cali{N}}$ for the Galois conjugate of $\cali{N}$. Suppose $\cali{N}$ is a prime ideal, $\cl_+(\cali{N})\ne \cl_+(F)$, and $H$ is a CM field. Then, up to $\mu_{p-1}p^{\Z}\subset \Z_p^\times$, by \cite[Corollary 1.10]{Zh23}, the left hand side of \eqref{eq:product-formula} can be recast as (see \S\ref{subsec:GS})
		\begin{align}\label{eq:norm}
			\nm_{\Q(\sqrt{3})\otimes \Q_p/\Q_p}\left(\frac{\prod_{\tau\in\gal(H/F),\tau|_{H_0}\ne 1} u_{\cali{N}}^\tau}{\prod_{\sigma\in\gal(H/F),\sigma|_{H_0}=1}u_{\cali{N}}^\sigma}\right)
			\prod_{\substack{y\in (\sqrt{3}-1)^{-1}\bar{\cali{N}}^{-1}\cap P(V)\\ y\notin \bar{\cali{N}}^{-1}\cap P(V)}} \Gamma_{p,V}(y)^{1-N},
		\end{align}
		where $u_{\cali{N}}$ is the Gross--Stark unit of $H/F$ for some fixed prime $\got{P}$ of $H$, known to exist by \cite{DDP,Ve15} and \cite[Proposition 3.8]{Gr81}. The product
		\begin{align*}
			\prod_{\substack{y\in (\sqrt{3}-1)^{-1}\bar{\cali{N}}^{-1}\cap P(V)\\ y\notin \bar{\cali{N}}^{-1}\cap P(V)}} \Gamma_{p,V}(y)
		\end{align*}
		is therefore also algebraic since $\Gamma_p(1/4),\Gamma_p(1/3)$ are, but does not have a direct arithmetic interpretation.
	\end{rem}
	
	\subsection{Organization of this paper} In \S\ref{sec:sumexpr}, we study some combinatorial properties of the periodic functions appearing in the explicit formulas of the measures attached to $p$-adic Shintani $L$-functions; this allows us to present a different form of sum expression than \cite[Corollary 3.8]{Zh23}. In \S\ref{sec:FG}, we explain how the computations performed in \S5 and Appendix A, \textit{op.~cit.}, together with the new sum expression, yield a proof of Theorem \ref{thm.A}. Finally, we apply this derivative formula to genus $L$-functions in \S\ref{sec:gamma-product} to establish the product formula \eqref{eq:product-formula}; we shall also numerically verify it for $p=5,7$ up to some $p$-adic precision. In the appendix, we consider two instances where an ``independent of $N$'' formula, much like \eqref{eq:product-formula}, results from the Euler regularization process, and we show both are arithmetic in nature. The second one concerning Morita Gamma values corresponds to the classical Davenport--Hasse relation \cite{DH35}.
	
	\subsection{Acknowledgment} I am thankful to Antonio Lei for his encouragement during the writing of this article. I am also grateful to him, Alan Lauder and Ari Shnidman for helpful comments. Part of this work is supported under the European Research Council (ERC, CurveArithmetic, 101078157).
	
	\section{Sum expression revisited}
	\label{sec:sumexpr}
	We will borrow the notation from \cite[\S2]{Zh23}. Below fix a large enough coefficient ring $R$ over $\Z_p$. If $h>1$ is an integer and $a\in \Z/h$, we denote by $a^\flat_h$ (resp.~$a^\sharp_h$) the unique integer in $[0,h)$ (resp.~$(0,h]$) that is congruent to $a$ modulo $h$. For the rest of this article, choose a Cassou-Nogu\`es ideal $\cali{N}\ne \cali{O}$ that is prime to $p$; write $N=\#\cali{O}/\cali{N}$. Fix a power $q>1$ of $p$ such that $q\equiv 1\bmod N$. Let $x\in F$ be $\cali{N}p$-integral and $V = \{v_1,\cdots,v_k\}$ be a subset of $\cali{O}$ that spans $\cali{O}_p:=\cali{O}\otimes \Z_p$; write $x = \sum_{1\le i\le k}x_iv_i$ for some $x_i\in \Q\cap \Z_p$. We will assume that $\gcd(v_i,\cali{N}) =1$ for all $i$. Let $\mu_{V,x,\cali{N}}$ be the $p$-adic measure on $\cali{O}_p$ corresponding to the rational function
	\begin{align*}
		f_{V,x,\cali{N}}(t_1,\cdots,t_k) = \sum_{\xi} \xi(x)\prod_{1\le i\le k}\frac{t_i^{x_i}}{1-\xi(v_i)t_i} \in R[[t_1-1,\cdots,t_k-1]]
	\end{align*}
	via the $p$-adic Fourier transform of Amice and Mazur, where the sum is over all additive characters $\xi:\cali{O}/\cali{N}\to \bar{\Q}^\times$ with $\xi\ne 1$. We refer the reader to \cite[\S3.8]{Hida} for details, as well as its relation with the measure attached to the $p$-adic $L$-functions regularized at $\cali{N}$. Now, we put
	\begin{align}\label{eq:partial-integral-representation}
		L_{p,V,x,\cali{N}}(s,\omega_F) = \int_{\cali{O}_p^\times}\chx{\nm (\alpha)}^{-s}\mu_{V,x,\cali{N}}(\alpha).
	\end{align}
	By the explicit period formula \cite[Theorem 3.6]{Zh23}, we have, for $n\in \Z_{\ge 0}$ and $0\le l_1,\cdots,l_k<q^n$,
	\begin{align*}
		\mu_{V,x,\cali{N}}(x+l\cdot v + q^n \cali{O}_p) = (-1)^k\got{h}_V(x+l\cdot v).
	\end{align*}
	Here, for a subset $U = \{u_1,\cdots,u_r\}\ne \emptyset$ of $\cali{O}$, $\got{h}_U$ stands for the periodic function on $\cali{O}/\cali{N}$:
	\begin{align*}
		\got{h}_U(a) = \frac{1}{N^{r-1}}\sum_{\substack{0\le d_1,\cdots,d_r<N\\ d_1u_1+\cdots+d_ru_r\equiv -a\bmod \cali{N}}}
		d_1d_2\cdots d_r - \left(\frac{N-1}{2}\right)^r.
	\end{align*}
	Taking the Riemann sum of \eqref{eq:partial-integral-representation}, we find
	\begin{align}\label{eq:riemann-sum}
		L_{p,V,x,\cali{N}}(s,\omega_F) = (-1)^k\lim_{n\to \infty} \sum_{\substack{0\le l<q^n\\ \gcd(p,x+l\cdot v)=1}} 
		\got{h}_V(x+l\cdot v)\chx{\nm(x+l\cdot v)}^{-s}.
	\end{align}
	
	In order to state the lemma below, we define
	\begin{align*}
		\got{h}_\emptyset(a) = \begin{cases}
			-1 & \text{if }a\not\equiv 0 \bmod \cali{N};\\
			N-1 & \text{if }a\equiv 0 \bmod \cali{N}.
		\end{cases}
	\end{align*}
	\begin{lem}\label{lem:h-identity}
		Let $U=\{u_1,\cdots,u_r\}\subset \cali{O}$ be nonempty, and suppose for all $1\le j\le r$, $\gcd(u_j,\cali{N})=1$. Then for $1\le i\le r$ and $a\in \cali{O}/\cali{N}$, we have
		\begin{align*}
			\got{h}_U(a+u_i) = \got{h}_U(a) + \got{h}_{U\setminus\{u_i\}}(a).
		\end{align*}
	\end{lem}
	\begin{proof}
		For $\emptyset\ne U' = \{u'_1,\cdots,u'_s\}\subset \cali{O}$ we put
		\begin{align*}
			H_{U'}(a) = \sum_{0\le d_1,\cdots,d_s<N, d\cdot u'\equiv -a\bmod \cali{N}}
			d_1\cdots d_s,
		\end{align*}
		so $\got{h}_{U'} = H_{U'}/N^{s-1} - (N-1)^s/2^s$. Without loss of generality assume $i=r$. When $r\ge 2$, we have
		\begin{align*}
			H_U(a+u_r) &=\sum_{0\le d<N, d\cdot u\equiv -a-u_r\bmod \cali{N}} d_1\cdots d_r\\
			&=\sum_{0\le d<N, d\cdot u\equiv -a\bmod \cali{N}} d_1\cdots d_{r-1}(d_r-1)^\flat_N\\
			&=\sum_{0\le d<N, d\cdot u\equiv -a\bmod \cali{N}} d_1\cdots d_{r-1}(d_r-1+N\1_{d_r=0})\\
			&=H_U(a) - \left(\frac{N(N-1)}{2}\right)^{r-1} + NH_{U\setminus\{u_r\}}(a).
		\end{align*}
		Here in the last equality, to get the term $\left(\frac{N(N-1)}{2}\right)^{r-1}$, we used the fact that the cyclic group $\cali{O}/\cali{N}$ is generated by $u_r$. Now, divide both sides by $N^{r-1}$, and we get
		\begin{align*}
			\got{h}_U(a+u_r) = \got{h}_U(a) - \left(\frac{N-1}{2}\right)^{r-1} + \got{h}_{U\setminus\{u_r\}}(a) + \left(\frac{N-1}{2}\right)^{r-1} = \got{h}_U(a) + \got{h}_{U\setminus\{u_r\}}(a).
		\end{align*}
		The case of $k=1$ can be verified separately.
	\end{proof}
	
	\begin{cor}
		Notation as in Lemma \ref{lem:h-identity}. For $x\in \cali{O}/\cali{N}$ and $l_1,\cdots,l_r\ge 0$, we have
		\begin{align*}
			\got{h}_U(x+l\cdot u) = \got{h}_U(x) + \sum_{\emptyset\ne S\subseteq \{1,2,\cdots,r\}} \sum_{\substack{0\le d_j<l_j\\ j\in S}} \got{h}_{U\backslash\{u_j\}_{j\in S}}\Bigg(x+\sum_{j\in S}d_ju_j\Bigg).
		\end{align*}
	\end{cor}
	\begin{proof}
		We prove by induction on $r=\#U$. When $r=1$, apply Lemma \ref{lem:h-identity} and we have
		\begin{align*}
			\begin{split}
				\got{h}_{U}(x+l_1u_1) &= \got{h}_U(x+(l_1-1)u_1) + \got{h}_\emptyset(x+(l_1-1)u_1)\\
				&=\cdots\\
				&=\got{h}_U(x) + \sum_{0\le d_1<l_1} \got{h}_\emptyset(x+d_1u_1).
			\end{split}
		\end{align*}
		For general $r$, iterating Lemma \ref{lem:h-identity}, we find
		\begin{align*}
			\begin{split}
				\got{h}_U(x+l\cdot u) &= \got{h}_U\Bigg(x+\sum_{0\le j<r}l_ju_j\Bigg) + \sum_{0\le d_r<l_r} \got{h}_{U\backslash\{u_r\}}\Bigg(x+\sum_{0\le j<r}l_ju_j +d_ru_r\Bigg)\\
				&=\cdots\\
				&=\got{h}_U(x) + \sum_{1\le i\le r} \sum_{0\le d_i<l_i}\got{h}_{U\backslash\{u_i\}}
				\Bigg(x+\sum_{0\le j<i}l_ju_j + d_iu_i\Bigg).
			\end{split}
		\end{align*}
		By induction hypothesis, we have for $1\le i\le r$ and $d_i\ge 0$,
		\begin{align*}
			\got{h}_{U\backslash\{u_i\}}\Bigg(x+\sum_{0\le j<i}l_ju_j + d_iu_i\Bigg) = \sum_{S\subseteq\{1,2,\cdots,i-1\}} \sum_{\substack{0\le d_j<l_j\\ j\in S}} \got{h}_{U\backslash\{u_j\}_{j\in S\cup\{i\}}}\Bigg(x+d_iu_i+\sum_{j\in S}d_ju_j\Bigg).
		\end{align*}
		Hence
		\begin{align*}
			\sum_{0\le d_i<l_i}\got{h}_{U\backslash\{u_i\}}\Bigg(x+\sum_{0\le j<i}l_ju_j + d_iu_i\Bigg) = \sum_{\{i\}\subseteq S\subseteq\{1,2,\cdots,i\}} \sum_{\substack{0\le d_j<l_j\\ j\in S}} \got{h}_{U\backslash\{u_j\}_{j\in S}}\Bigg(x+\sum_{j\in S}d_ju_j\Bigg).
		\end{align*}
		Note that if $S$ is a nonempty subset of $\{1,2,\cdots,r\}$, then there is a unique $1\le i\le r$ such that $S\subseteq\{1,2,\cdots,i\}$ and $S\supseteq \{i\}$, i.e., $i=\max S$. It follows that
		\begin{align*}
			\sum_{1\le i\le r} \sum_{0\le d_i<l_i}\got{h}_{U\backslash\{u_i\}}
			\Bigg(x+\sum_{0\le j<i}l_ju_j + d_iu_i\Bigg) = \sum_{\emptyset\ne S\subseteq\{1,2,\cdots,r\}}\sum_{\substack{0\le d_j<l_j\\ j\in S}}\got{h}_{U\backslash\{u_j\}_{j\in S}}\Bigg(x+\sum_{j\in S}d_ju_j\Bigg).
		\end{align*}
	\end{proof}
	
	\begin{cor}[Sum expression, second form]
		We have
		\begin{align*}
			L_{p,V,x,\cali{N}}(s,\omega_F) = (-1)^k\lim_{n\to\infty} \sum_{\substack{0\le l<q^n\\ \gcd(p,x+l\cdot v)=1}}
			\sum_{0\le d<l}\got{h}_\emptyset(x+d\cdot v) \chx{\nm(x+l\cdot v)}^{-s}.
		\end{align*}
	\end{cor}
	\begin{proof}
		If $S\ne \{1,2,\cdots,k\}$, $\sum_{0\le d_j\le l_j, j\in S} \got{h}_{V\backslash\{v_j\}_{j\in S}}(x+\sum_{j\in S} d_jv_j)$ is independent of $l_i$ for any $i\notin S$. The continuity of the function $\chx{\nm(\cdot)}^{-s}$ then deems the relevant terms in \eqref{eq:riemann-sum} vanish as $n$ approaches $\infty$ (\textit{cf}.~\cite[Lemma 4.6]{Zh23}).
	\end{proof}
	\begin{rem}
		One should compare the above formula with equation (4.5) \textit{ibid}.
	\end{rem}
	
	\section{Derivative formula}
	\label{sec:FG}
	In the rest of this article, we will further suppose that, for all $V$ in the Shintani decomposition,
	\begin{align*}
		\cali{O}_p = \sum_{1\le i\le k} \Z_p v_i.
	\end{align*}
	By essentially the same proof of \cite[Proposition 5.2]{Zh23}, exploiting the fact that the function $\got{h}_\emptyset$ is $N$-periodic and is such that $\sum_{a\in \cali{O}/\cali{N}}\got{h}_\emptyset(a) = 0$, we obtain
	\begin{thm}\label{thm.FG}
		We have
		\begin{align*}
			L_{p,V,x,\cali{N}}'(0,\omega_F) = (-1)^{k-1}\sum_{0\le d<N}\got{h}_\emptyset(x+d\cdot v)\log_p\Gamma_{p,V}\left(\frac{x+d\cdot v}{N}\right) - k\log_p N L_{p,V,x,\cali{N}}(0,\omega_F).
		\end{align*}
	\end{thm}
	Next, for $y\in F_+$ that is $\cali{N}$-integral, consider the complex Shintani $L$-function
	\begin{align*}
		L_{V,y,\cali{N}}(s) = \sum_{l_1,\cdots,l_k\ge 0} \frac{\got{h}_\emptyset(y+l\cdot v)}{\nm(y+l\cdot v)^s}.
	\end{align*}
	To state the next proposition we recycle some notation from \cite{Zh23}. Let $C(V) = \sum_{1\le i\le k}\R_{>0}v_i$, $\overline{C}(V)$ be its upper closure and $P(V)$ the fundamental parallelotope of $\overline{C}(V)$. If $\got{a}$ is a prime-to-$p$ integral ideal then we denote by $\tau_p$ the isomorphism on $\got{a}^{-1}\cap P(V)$ corresponding to the $p$-multiplication on the $p$-divisible group $\got{a}^{-1}/\Z\cdot V$.
	\begin{lem}\label{lem:interpolation-zero}
		Assume $p$ is inert in $F$. Let $\got{a}$ be a prime-to-$p$ integral ideal and $x\in \got{a}^{-1}\cap P(V)$. We have the interpolation formula
		\begin{align*}
			L_{p,V,x,\cali{N}}(0) = L_{V,x,\cali{N}}(0) - L_{V,\tau_p^{-1}x,\cali{N}}(0).
		\end{align*}
	\end{lem}
	\begin{proof}
		The proof proceeds in the same way in Appendix A \textit{op.~cit.}; the only difference is that instead of having $\chi(pa) = \chi(p)\chi(a)$, we have $\got{h}_\emptyset(pa) = \got{h}_\emptyset(a)$. Note that the assumption that $p$ is inert is used crucially in the computation; see line 3-5 of p.~628 \textit{ibid}.
	\end{proof}
	
	\begin{proof}[Proof of Theorem \ref{thm.A}]
		By Lemma \ref{lem:interpolation-zero}, for any $1\le i\le h$, we have
		\begin{align*}
			\sum_{x\in P(V)\cap\got{a}_i^{-1}} L_{p,V,x,\cali{N}}(0) = \sum_{x\in P(V)\cap\got{a}_i^{-1}} L_{V,x,\cali{N}}(0) - \sum_{x\in P(V)\cap\got{a}_i^{-1}} L_{V,\tau_p^{-1}x,\cali{N}}(0) = 0.
		\end{align*}
		As
		\begin{align*}
			(1-\chi\omega_F(\cali{N})\chx{N}^{1-s})L_{F,p}(s,\chi\omega_F) = -\sum_{1\le i\le h}\chi(\got{a}_i)\chx{\nm(\got{a}_i)}^{-s}\sum_{V}\sum_{x\in P(V)\cap\got{a}_i^{-1}}L_{p,V,x,\cali{N}}(s),
		\end{align*}
		taking the derivative and using Theorem \ref{thm.FG}, we find
		\begin{align*}
			(1-\chi(\cali{N})N)L'_{F,p}(0,\chi\omega_F) = (-1)^k\sum_{i,V,x}\chi(\got{a}_i)\sum_{0\le d<N}\got{h}_\emptyset(x+d\cdot v)\log_p\Gamma_{p,V}\left(\frac{x+d\cdot v}{N}\right).
		\end{align*}
	\end{proof}

	\section{A product formula of bivariate $p$-adic Gamma values}
	\label{sec:gamma-product}
	\subsection{Background on genus $L$-functions}
	
	From now on $F$ will be a real quadratic field with discriminant $D$ prime to $p$. Recall a number $d$ is called a fundamental discriminant if $d=1$ or $d$ is the discriminant of a quadratic field. Suppose there is a factorization $D = D_1D_2$ where both $D_1,D_2$ are fundamental discriminants with $\gcd(D_1,D_2)=1$, then the genus character $\chi_{D_1,D_2}:\cl_+(F)\to \{\pm 1\}$ is the character whose kernel cuts out the extension $\Q(\sqrt{D_1},\sqrt{D_2})/F$. Furthermore, denoting by $\chi_{D_1}$ and $\chi_{D_2}$ the Dirichlet characters of quadratic fields $\Q(\sqrt{D}_1)$ and $\Q(\sqrt{D_2})$ respectively, then we have a factorization of complex $L$-functions:
	\begin{align}\label{eq:genus}
		L_F(s,\chi_{D_1,D_2}) = L(s,\chi_{D_1})L(s,\chi_{D_2}).
	\end{align}

	\subsection{The product formula}
	
	We now prove the following
	\begin{prop}\label{prop.prod-formula}
		Let $F$ be a real quadratic field whose discriminant $D$ admits a factorization $D = D_1D_2$ into fundamental discriminants $D_1,D_2<0$. Assume $p$ is inert in $F$, and let $i\in\{1,2\}$ be the unique integer such that $p$ is inert in $\Q(\sqrt{D_i})$; let $\{j\}=\{1,2\}\setminus\{i\}$. Then we have
		\begin{align}\label{eq:gamma-factorization}
			\begin{split}
				&(1-\chi(\cali{N})N)^{-1}\sum_{1\le i\le h}\chi(\got{a}_i)\sum_{x\in P(V)\cap\got{a}_i^{-1}}\sum_{0\le d<N}\got{h}_\emptyset(x+d\cdot v)\log_p\Gamma_{p,V}\left(\frac{x+d\cdot v}{N}\right)\\
				=& -2\sum_{0\le e_i<|D_i|}\chi_{D_i}(e_i)\frac{e_i}{|D_i|}\sum_{0\le e_j<|D_j|}\chi_{D_j}(e_j)\log_p\Gamma_p\left(\frac{e_j}{|D_j|}\right).
			\end{split}
		\end{align}
	\end{prop}
	\begin{proof}
		As $p$ is inert in $F$, $\chi_{D_1,D_2}(p) =1$ and thus $p$ is inert in one of $\Q(\sqrt{D_1}),\Q(\sqrt{D_2})$ and split in the other. Rearranging if necessary, we suppose $p$ is inert in $\Q(\sqrt{D_1})$. Now, the complex factorization \eqref{eq:genus} implies the $p$-adic one
		\begin{align*}
			L_{F,p}(s,\chi_{D_1,D_2}\omega_F) = L_p(s,\chi_{D_1}\omega)L_p(s,\chi_{D_2}\omega).
		\end{align*}
		The splitting assumption then implies $\ord_{s=0}L_p(s,\chi_{D_2}\omega)=1$ \cite[Proposition 2]{FG78}. Taking the derivative and applying Theorem \ref{thm.A} and the Ferrero--Greenberg formula, we have the desired equality; note that the factor $2$ in the second line comes from the Euler factor $1-\chi_{D_1}(p) = 2$.
	\end{proof}

	\begin{rem}
		It is worth pointing out that, as $p\nmid D$, the condition that $p$ is inert is equivalent to that $L_{F,p}(s,\chi_{D_1,D_2}\omega_F)$ has exactly vanishing order 1, namely
		\begin{align*}
			\ord_{s=0}L_p(s,\chi_{D_1}\omega) + \ord_{s=0}L_p(s,\chi_{D_2}\omega) =1.
		\end{align*}
		To see this, first recall that for an odd Dirichlet character $\psi$ of conductor prime to $p$, $L_p(s,\psi\omega)$ vanishes if and only if $\psi(p)=1$; if that is the case, the derivative is nonzero by formulas of Ferrero--Greenberg \cite{FG78} and Gross--Koblitz \cite{GK79}, eventually thanks to Brumer's independence \cite{Br67}. Now, if $p$ is split, let $\got{p}$ be a prime above $p$, and we find \cite[equation (12.59)]{Iwaniec}
		\begin{align*}
			\chi_{D_1}(p) = \chi_{D_2}(p) = \chi_{D_1,D_2}(\got{p}).
		\end{align*}
		Therefore $\ord_{s=0}L_{F,p}(s,\chi_{D_1,D_2}\omega_F)$ is $0$ or $2$.
	\end{rem}

	\subsection{The case of $\Q(\sqrt{3})$}\label{subsec:sqrt3}
	
	We now let $F = \Q(\sqrt{3})$, which has discriminant $12$. Then, we have a factorization of $12$ into the product of $-3$ and $-4$, the discriminants of $\Q(\sqrt{-3})$ and $\Q(i)$. We set the stage as follows:
	\begin{enumerate}
		\item Let $p$ be inert in $\Q(\sqrt{3})$, so $p\equiv 5,7\bmod 12$.
		\item With $\cl_+(F) = \Z/2$, we fix a set of representatives to be $\{\cali{O},(\sqrt{3}-1)\}$.
		\item Choose a Cassou-Nogu\`es ideal $\cali{N}\ne \cali{O}$ prime to $2p$; e.g., $\cali{N} = (\sqrt{3})$.
		\item Let $\varepsilon = 2+\sqrt{3}$. Then $\cali{O}^\times_+ = \varepsilon^{\Z}$ and the Shintani decomposition for $F$ is given by a single cone $V = \{1,\varepsilon\}$. We also have $\cali{O} = \Z+\Z\varepsilon$, so $\cali{O}_p = \Z_p+\Z_p\varepsilon$. Note that $P(V)\cap\cali{O}$ consists of $1$ and $P(V)\cap(\sqrt{3}-1)^{-1}$ consists of $1$ and $\frac{1+\varepsilon}{2} = \frac{3+\sqrt{3}}{2}$.
	\end{enumerate}
	\begin{proof}[Proof of Corollary \ref{cor:B}]
		This will be a simple verification using \eqref{eq:gamma-factorization}. Write $\chi = \chi_{-3,-4}$ for simplicity. We see that:
		\begin{align*}
			&(1-\chi(\cali{N})N)^{-1}\sum_{i,x}\chi(\got{a}_i)\sum_{0\le d<N}\got{h}_\emptyset(x+d\cdot v)\log_p\Gamma_{p,V}\left(\frac{x+d\cdot v}{N}\right)\\
			=&-(1-\chi_{-4}(N)N)^{-1}\sum_{0\le d<N}\got{h}_\emptyset\Big(\frac{1+\varepsilon}{2}+d\cdot v\Big)\log_p\Gamma_{p,V}\left(\frac{\frac{1+\varepsilon}{2}+d\cdot v}{N}\right)\\
			=&\frac{1}{1-\chi_{-4}(N)N}\log_p\left(\frac{\prod_{0\le a,b<N}\Gamma_{p,V}\left(\frac{\frac{1+\varepsilon}{2}+a+b\varepsilon}{N}\right)}{\prod_{a+b\varepsilon\equiv -(1+\varepsilon)/2 \bmod \cali{N}} \Gamma_{p,V}\left(\frac{\frac{1+\varepsilon}{2}+a+b\varepsilon}{N}\right)^N}\right)
		\end{align*}
		Now, if $p\equiv 5\bmod 12$, i.e., $p$ is inert in $\Q(\sqrt{-3})$, then we have $\sum_{0\le e_1<3}\chi_{-3}(e_1)\frac{e_1}{3}= -\frac{1}{3}$, and the right hand side of \eqref{eq:gamma-factorization} reads
		\begin{align*}
			\frac{2}{3}\sum_{0\le e_2<4}\chi_{-4}(e_2)\log_p\Gamma_p(e_2/4) = \frac{2}{3}\log_p\frac{\Gamma_p(1/4)}{\Gamma_p(3/4)}= \frac{4}{3}\log_p\Gamma_p(1/4),
		\end{align*}
		where the last equality follows from the product formula $\Gamma_p(z)\Gamma_p(1-z)=\pm 1$ for any $z\in \Z_p$.
		
		Similarly, if $p\equiv 7\bmod 12$, then the right hand side of \eqref{eq:gamma-factorization} becomes $4\log_p\Gamma_p(1/3)$. Exponentiate and we are done.
	\end{proof}
	\comment{\begin{rem}
		By varying the choice of representatives of $\cl_+(F)$, it is possible to obtain other product formulas for $\Q(\sqrt{3})$, which are likely more complicated due to the lack of control on the discrete sets $\got{a}^{-1}\cap P(V)$. For example, if $\gcd(3,N)=1$, then replacing $(\sqrt{3}-1)$ by $(\sqrt{3})$, the left hand side of \eqref{eq:product-formula} becomes
		\begin{align*}
			\frac{\prod_{0\le a,b<N}\Gamma_{p,V}\left(\frac{\frac{1+\varepsilon}{3}+a+b\varepsilon}{N}\right)\Gamma_{p,V}\left(\frac{\frac{2+2\varepsilon}{3}+a+b\varepsilon}{N}\right)}{\prod_{a+b\varepsilon\equiv -(1+\varepsilon)/3 \bmod \cali{N}} \Gamma_{p,V}\left(\frac{\frac{1+\varepsilon}{3}+a+b\varepsilon}{N}\right)^N\prod_{a+b\varepsilon\equiv -(2+2\varepsilon)/3 \bmod \cali{N}} \Gamma_{p,V}\left(\frac{\frac{2+2\varepsilon}{3}+a+b\varepsilon}{N}\right)^N}.
		\end{align*}
	\end{rem}}
	\begin{eg}
		We now numerically verify the product formula \eqref{eq:product-formula} for $\cali{N}=(\sqrt{3})$ and $p=5,7$ up to certain $p$-adic precision. In this case, using the symmetry $\Gamma_{p,V}(a,b) = \Gamma_{p,V}(b,a)$, we have a simpler formula:
		\begin{align*}
			\frac{\Gamma_{p,V}\left(\frac{1+3\varepsilon}{6}\right)\Gamma_{p,V}\left(\frac{1+5\varepsilon}{6}\right)\Gamma_{p,V}\left(\frac{3+5\varepsilon}{6}\right)}{\Gamma_{p,V}\left(\frac{1+\varepsilon}{6}\right)\Gamma_{p,V}\left(\frac{3+3\varepsilon}{6}\right)\Gamma_{p,V}\left(\frac{5+5\varepsilon}{6}\right)}
			=\begin{cases}
				\chx{\Gamma_p(1/4)}^{8/3}&\text{if }p\equiv 5\bmod 12;\\
				\chx{\Gamma_p(1/3)}^4&\text{if }p\equiv 7\bmod 12.
			\end{cases}
		\end{align*}
		Consider first $p=5$. We have
		\begin{align*}
			&\Gamma_{5,V}\left(\frac{1+3\varepsilon}{6}\right) \equiv 411, \qquad \Gamma_{5,V}\left(\frac{1+5\varepsilon}{6}\right) \equiv 81, \qquad 
			\Gamma_{5,V}\left(\frac{3+5\varepsilon}{6}\right) \equiv 86, \\
			&\Gamma_{5,V}\left(\frac{1+\varepsilon}{6}\right) \equiv 146, \qquad
			\Gamma_{5,V}\left(\frac{3+3\varepsilon}{6}\right) \equiv 441, \qquad
			\Gamma_{5,V}\left(\frac{5+5\varepsilon}{6}\right) \equiv 496\pmod{625}.
		\end{align*}
		As such,
		\begin{align*}
			\left(\frac{\Gamma_{5,V}\left(\frac{1+3\varepsilon}{6}\right)\Gamma_{5,V}\left(\frac{1+5\varepsilon}{6}\right)\Gamma_{5,V}\left(\frac{3+5\varepsilon}{6}\right)}{\Gamma_{5,V}\left(\frac{1+\varepsilon}{6}\right)\Gamma_{5,V}\left(\frac{3+3\varepsilon}{6}\right)\Gamma_{5,V}\left(\frac{5+5\varepsilon}{6}\right)}\right)^{3/8} \equiv 
			\left(\frac{411\times 81\times 86}{146\times 441\times 496}\right)^{3/8}\equiv 21\bmod 625,
		\end{align*}
		and we do have
		\begin{align*}
			\chx{\Gamma_5(1/4)} \equiv \left\langle\Gamma_5\left(\frac{1+3\times 5^7}{4}\right)\right\rangle \equiv 21\bmod 625.
		\end{align*}
		
		Next consider $p=7$, and we have
		\begin{align*}
			&\Gamma_{7,V}\left(\frac{1+3\varepsilon}{6}\right) \equiv 260, \qquad 
			\Gamma_{7,V}\left(\frac{1+5\varepsilon}{6}\right) \equiv 211, \qquad 
			\Gamma_{7,V}\left(\frac{3+5\varepsilon}{6}\right) \equiv 218, \\
			&\Gamma_{7,V}\left(\frac{1+\varepsilon}{6}\right) \equiv 190, \qquad
			\Gamma_{7,V}\left(\frac{3+3\varepsilon}{6}\right) \equiv 288, \qquad
			\Gamma_{7,V}\left(\frac{5+5\varepsilon}{6}\right) \equiv 204\pmod{343}.
		\end{align*}
		As such,
		\begin{align*}
			\left(\frac{\Gamma_{7,V}\left(\frac{1+3\varepsilon}{6}\right)\Gamma_{7,V}\left(\frac{1+5\varepsilon}{6}\right)\Gamma_{7,V}\left(\frac{3+5\varepsilon}{6}\right)}{\Gamma_{7,V}\left(\frac{1+\varepsilon}{6}\right)\Gamma_{7,V}\left(\frac{3+3\varepsilon}{6}\right)\Gamma_{7,V}\left(\frac{5+5\varepsilon}{6}\right)}\right)^{1/4} \equiv 
			\left(\frac{260\times 211\times 218}{190\times 288\times 204}\right)^{1/4} \equiv 15 \bmod 343,
		\end{align*}
		whereas
		\begin{align*}
			\chx{\Gamma_7(1/3)}\equiv \left\langle\Gamma_7\left(\frac{1+2\times 7^5}{3}\right)\right\rangle \equiv 15\bmod 343,
		\end{align*}
		as expected.
	\end{eg}
	\begin{rem}
		Here, we only verified the formula for some low $p$-adic precision because of the slow convergence of the formula defining $p$-adic Gamma values. It can be shown that if $(m,n)\equiv (m',n')\bmod p^{2n-1}$, then $\Gamma_{p,V}(m,n) \equiv \Gamma_{p,V}(m',n')\bmod p^n$. As such, to approximate a bivariate Gamma value up to $p^n$, we need to perform $O(p^{4n-2})$ operations as there are two variables.
	\end{rem}
	
	\subsection{Connection with Gross--Stark units}\label{subsec:GS}
	
	Keep the same notation and assumptions as in \S\ref{subsec:sqrt3} and Remark \ref{rem:Gross-Stark}. In particular, $\cali{N}$ will now denote a Cassou-Nogu\`es prime. Below, we will validate the claim in Remark \ref{rem:Gross-Stark} via the following
	\begin{prop}\label{prop:arithmetic}
		Let $\tau\in\gal(H/F)$ be such that $\tau|_{H_0}\ne 1$. We have
		\begin{align*}
			\frac{\prod_{0\le a,b<N}\Gamma_{p,V}\left(\frac{\frac{1+\varepsilon}{2}+a+b\varepsilon}{N}\right)}{\prod_{a+b\varepsilon\equiv -(1+\varepsilon)/2 \bmod \cali{N}} \Gamma_{p,V}\left(\frac{\frac{1+\varepsilon}{2}+a+b\varepsilon}{N}\right)^N}
			\sim \frac{\nm_{\Q(\sqrt{3})\otimes \Q_p/\Q_p}\left(\prod_{\sigma\in\gal(H/H_0)}u_{\cali{N}}^{\tau\sigma}/u_{\cali{N}}^\sigma\right)}{\prod_{\substack{y\in (\sqrt{3}-1)^{-1}\bar{\cali{N}}^{-1}\cap P(V)\\ y\notin \bar{\cali{N}}^{-1}\cap P(V)}} \Gamma_{p,V}(y)^{N-1}},
		\end{align*}
		where $\sim$ means an equality in $\Z_p$ up to $\mu_{p-1}p^{\Z}$.
	\end{prop}
	
	We will need a lemma:
	\begin{lem}\label{lem:fun-domain}
		The set $(\sqrt{3}-1)^{-1}\bar{\cali{N}}^{-1}\cap P(V)\setminus\bar{\cali{N}}^{-1}\cap P(V)$ admits an explicit parametrization
		\begin{align*}
			\left\{\frac{\frac{1+\varepsilon}{2}+a+b\varepsilon}{N}:0\le a,b<N,\frac{1+\varepsilon}{2}+a+b\varepsilon\equiv 0\bmod \cali{N}\right\}.
		\end{align*}
	\end{lem}
	\begin{proof}
		We first show the former set is contained in the latter. Let $x\in (\bar{\cali{N}}^{-1}(\sqrt{3}-1)^{-1}\setminus\bar{\cali{N}}^{-1})\cap P(V)$, then $Nx\in (\sqrt{3}-1)^{-1}\cap \overline{C}(V)$. So there exists $a,b\in \Z_{\ge 0}$ such that either $Nx = 1+a+b\varepsilon$ or $Nx = \frac{1+\varepsilon}{2}+a+b\varepsilon$.
		
		Suppose $Nx = 1 + a+b\varepsilon$, then $Nx\in \cali{O}$, so $x \in (N)^{-1}$. It follows that $x\in (N)^{-1}\cap \bar{\cali{N}}^{-1}(\sqrt{3}-1)^{-1} = \bar{\cali{N}}^{-1}$, by the unique factorization of ideals. This contradicts the assumption that $x\notin \bar{\cali{N}}^{-1}$.
		
		Thus $Nx = \frac{1+\varepsilon}{2} + a + b\varepsilon$. As $N$ is odd, we have an identification $(\sqrt{3}-1)^{-1}/(\sqrt{3}-1)^{-1}\cali{N}= \cali{O}/\cali{N}$. Therefore, since $Nx \in (\sqrt{3}-1)^{-1}\cali{N}$, we must have $\frac{1+\varepsilon}{2}+ a+ b\varepsilon\equiv 0\bmod \cali{N}$. This shows $x$ belongs to the second set.
		
		Now we tackle the inverse inclusion, so let $y = (\frac{1+\varepsilon}{2}+a+b\varepsilon)/N$ be in the latter set. By direct computation, $Ny(\sqrt{3}-1) = -(2+3a+b)+(a+b+1)\varepsilon \in \cali{O}$. Moreover, as $Ny\equiv 0\bmod \cali{N}$, we have $Ny(\sqrt{3}-1) \in\cali{N}$. This shows $y\in (\sqrt{3}-1)^{-1}\bar{\cali{N}}^{-1}$. To conclude the proof, we note that $Ny(\sqrt{3}-1)\notin (\sqrt{3}-1)\cali{N}$, as
		\begin{align*}
			\nm_{\Q(\sqrt{3})/\Q}\left(Ny(\sqrt{3}-1)\right)= (2+3a+b)^2 -4(2+3a+b)(a+b+1)+(a+b+1)^2 \equiv 1\bmod 2.
		\end{align*}  
	\end{proof}
	\begin{proof}[Proof of Proposition \ref{prop:arithmetic}]
		Firstly, note that there is a typo in \cite[equation (1.3)]{Zh23}, which should be corrected as ($u_{\cali{N}}$ is denoted as $u_p$ \textit{loc.~cit}.)
		\begin{align}\label{eq:gen-GK}
			\log_p\nm_{F\otimes \Q_p/\Q_p}u_{\cali{N}}^{\sigma_{\got{a}}} = (-1)^k\sum_{e\in E_{\cali{N}}}\sum_{0\le j<\nu}\sum_{x\in P(V)\cap\got{a}_i^{-1}}
			\sum_{\substack{0\le d_1,\cdots,d_k<N\\ x+d\cdot v\equiv ep^j y\bmod \cali{N}}} \Gamma_{p,V}\left(\frac{x+d\cdot v}{N}\right),
		\end{align}
		where $E_{\cali{N}}$ is the image of totally positive units of $\cali{O}$ in $(\cali{O}/\cali{N})^\times$ via reduction. The product on $e$ and $j$ comes from the following diagram describing the structure of $\cl_+(\cali{N})/(p)^{\Z}$ where rows are exact (\textit{cf.}~\cite[p.~102]{Hida}):
		\[
			\begin{tikzcd}
				0 \ar[r] & E_{\cali{N}} \ar[r] \ar[d] & (\cali{O}/\cali{N})^\times \ar[r]\ar[d] & \cl_+(\cali{N}) \ar[r] \ar[d] & \cl_+(F) \ar[r] \ar[d,equal]& 0\\
				0 \ar[r] & E_{\cali{N}}/p^{\Z}\cap E_{\cali{N}} \ar[r] & (\cali{O}/\cali{N})^\times/p^{\Z} \ar[r] & \cl_+(\cali{N})/(p)^{\Z} \ar[r] & \cl_+(F) \ar[r] & 0
			\end{tikzcd}.
		\]
		
		Now, the assumption that $\cali{N}$ is a prime means that if $a\in \cali{O}/\cali{N}$ is nonzero, then $a\in(\cali{O}/\cali{N})^\times$. As such, by \eqref{eq:gen-GK}, we have two identities up to $\mu_{p-1}p^{\Z}$
		\begin{align*}
			\prod_{\substack{\tau\in\gal(H/F)\\ \tau|_{H_0}\ne 1}} \nm_{\Q(\sqrt{3})\otimes\Q_p/\Q_p}(u_{\cali{N}}^\tau)
			\sim
			\prod_{\substack{0\le a,b<N\\ a+b\varepsilon\not\equiv -1\bmod\cali{N} }} \Gamma_{p,V}\left(\frac{1+a+b\varepsilon}{N}\right)
			\prod_{\substack{0\le a,b<N\\ a+b\varepsilon\not\equiv -\frac{1+\varepsilon}{2}\bmod\cali{N} }} \Gamma_{p,V}\left(\frac{\frac{1+\varepsilon}{2}+a+b\varepsilon}{N}\right)
		\end{align*}
		and
		\begin{align*}
			\prod_{\substack{\sigma\in\gal(H/F)\\ \tau|_{H_0}= 1}} \nm_{\Q(\sqrt{3})\otimes\Q_p/\Q_p}(u_{\cali{N}}^\sigma)
			\sim
			\prod_{\substack{0\le a,b<N\\ a+b\varepsilon\not\equiv -1\bmod\cali{N} }} \Gamma_{p,V}\left(\frac{1+a+b\varepsilon}{N}\right).
		\end{align*}
		Taking the quotient and employing Lemma \ref{lem:fun-domain}, we get the desired result.
	\end{proof}
	
	\appendix
	
	\section{Arithmetic properties arising from the variance of $N$}
	
	One remarkable feature of the product formula \eqref{eq:product-formula} is the freedom in choosing the Cassou-Nogu\`es ideal $\cali{N}$; in doing so, the only change on the right hand side is the Euler regularization factor in the exponent. In other words,
	\begin{align*}
		\frac{1}{1-\chi_{-4}(N)N}\log_p\left(\frac{\prod_{0\le a,b<N}\Gamma_{p,V}\left(\frac{\frac{1+\varepsilon}{2}+a+b\varepsilon}{N}\right)}{\prod_{a+b\varepsilon\equiv -(1+\varepsilon)/2 \bmod \cali{N}} \Gamma_{p,V}\left(\frac{\frac{1+\varepsilon}{2}+a+b\varepsilon}{N}\right)^N}\right)
	\end{align*}
	is independent of $\cali{N}$. We speculate that this independence could be explained by certain distribution relations of the Gross--Stark units on the arithmetic side. While it is not possible to verify this speculation due to the lack of an arithmetic interpretation of the product $\prod_{a+b\varepsilon\equiv -(1+\varepsilon)/2 \bmod \cali{N}} \Gamma_{p,V}\left(\frac{\frac{1+\varepsilon}{2}+a+b\varepsilon}{N}\right)$ as explained in Remark \ref{rem:Gross-Stark}, here we instead inspect two well-studied instances where similar independence arises from the Euler regularization; their arithmetic counterparts are product relations of cyclotomic units and of Gauss sums (after Davenport--Hasse), respectively.
	
	\subsection{First instance: cyclotomic units}
	
	Let $\chi$ be a nontrivial Dirichlet character of conductor $M$, and $L_p(s,\chi)$ be the Kubota--Leopoldt $p$-adic $L$-function. Let $N>1$ be an integer prime to $pM$. The formula of Leopoldt \cite[\S5.4]{Iw} asserts that, for a fixed primitive $M$-th root of unity $\zeta$,
	\begin{align*}
		L_p(1,\chi) = (1-\chi(N))^{-1}\frac{\tau(\chi)}{M}(1-\chi(p)/p)\sum_{0\le a<M} \chi^{-1}(a)\sum_{\lambda^N=1,\lambda\ne 1} \log_p(1-\lambda \zeta^a).
	\end{align*}
	Here $\tau(\chi) = \sum_{a\in (\Z/M)^\times} \chi(a)\zeta^a$ is the classical Gauss sum. As one can verify (see p.~60, \textit{ibid.}), the independence of $N$ of the right hand side is a consequence of the product formula of cyclotomic units
	\begin{align*}
		\prod_{\lambda^N=1,\lambda\ne 1}(1-\lambda\zeta^a) = \frac{1-\zeta^{aN}}{1-\zeta^a}.
	\end{align*}
	
	\subsection{Second instance: Gauss sums}
	
	Below we will first prove a product formula \eqref{eq:gamma-distribution} of Morita $p$-adic Gamma values by comparing a derivative formula of the regularized Kubota--Leopoldt $p$-adic $L$-function to that of Ferrero--Greenberg. We then use the Gross--Koblitz formula to show that the product formula translates to the Davenport--Hasse distribution relation of Gauss sums.

	\subsubsection{Measure computation}
	Let $M>1$ be an integer and $\chi: (\Z/M)^\times\to \bar{\Q}^\times$ be a nontrivial Dirichlet character. Let $N>1$ be an auxiliary integer that is prime to both $p$ and $M$. Denote by $L_p(s,\chi\omega)$ the Kubota--Leopoldt $p$-adic $L$-function that interpolates $(1-\chi(p)p^{k-1})L(1-k,\chi)$ for all $k\ge 1$, $k\equiv 1\bmod p-1$. 
	
	Consider the power series
	\begin{align*}
		f_{\chi,N}(t) = \sum_{0\le a<M} \chi(a)\left[\frac{t^a}{1-t^M} - N\frac{t^{(a/N)^\flat_M N}}{1-t^{MN}}\right] \in R[[t-1]].
	\end{align*}
	We claim that the measure $\mu_{\chi,N}$ on $\Z_p$, which is the Fourier inverse of $f_{\chi,N}$, gives an integral representation
	\begin{align*}
		L_{p,N}(s,\chi\omega):=(1-\chx{N}^{1-s}\chi\omega(N)) L_p(s,\chi\omega) = \int_{\Z_p^\times} \chx{x}^{-s}\mu_{\chi,N}(x).
	\end{align*}
	To see this, note that for $f_{\chi,N}^{(p)}(t) = f_{\chi,N}(t) - \frac{1}{p}\sum_{\zeta:\zeta^p=1}f_{\chi,N}(\zeta t)$, the following interpolation holds for all $k\ge 1$ and $k\equiv 1\bmod p-1$:
	\begin{align*}
		\left(t\frac{d}{dt}\right)^{k-1}f_{\chi,N}^{(p)}\Big|_{t=1} &= (1 - N^k\chi(N))
		\sum_{n\ge 1, p\nmid n}\chi(n)n^{k-1}t^{k-1}\Bigg|_{t=1}\\
		&=(1-N^k\chi(N))L_p(1-k,\chi\omega).
	\end{align*}
	Therefore
	\begin{align*}
		\int_{\Z_p^\times}\chx{x}^{-s}\mu_{\chi,N}(x) &= \lim_{\substack{k \to 1-s\\ k\equiv 1\bmod p-1}} \int_{\Z_p^\times}x^{k-1}\mu_{\chi,N}(x) = 
		\lim_{\substack{k \to 1-s\\ k\equiv 1\bmod p-1}}
		\left(t\frac{d}{dt}\right)^{k-1}f_{\chi,N}^{(p)}\Big|_{t=1}\\
		&=\lim_{\substack{k \to 1-s\\ k\equiv 1\bmod p-1}}(1-N^k\chi(N))L_p(1-k,\chi\omega) = (1-\chx{N}^{1-s}\chi\omega(N))L_p(s,\chi\omega).
	\end{align*}
	
	Next, we will compute some explicit period formulas of $\mu_{\chi,N}$. Write $f_\chi(t) = \sum_{0\le a<M} \chi(a)\frac{t^a}{1-t^M}$, then we have $f_{\chi}(t)\in R[[t-1]]$ (see, e.g., \cite[\S2]{Zh22}) and $f_{\chi,N}(t) = f_\chi(t) - \chi(N)N f_\chi(t^N)$. Denote by $\mu_\chi$ the corresponding measure of $f_\chi$. By Theorem 3.3 \textit{op.~cit.}, if $p^n\equiv 1\bmod M$, then for all $0\le a<p^n$, we have
	\begin{align*}
		\mu_{\chi}(a+p^n\Z_p) = L(0,\chi) - \sum_{0\le d<a}\chi(d).
	\end{align*}
	We then have
	\begin{align}\label{eq:1}
		\begin{split}
			\mu_{\chi,N}(a+p^n\Z_p) &= \mu_\chi(a+p^n\Z_p) - \chi(N)N\mu_\chi(a/N+p^n\Z_p)\\
			&= -\sum_{0\le d<a}\chi(d) + \chi(N)N\sum_{0\le d<(a/N)^\flat_{p^n}}\chi(d) + (1-\chi(N)N)L(0,\chi).
		\end{split}
	\end{align}
	\comment{To state the next lemma, recall that if $p^n\equiv 1\bmod N$, then there is an involution $\iota_N$ on the set $\{1\le m<p^n: p\nmid m\}$ due to Ferrero--Greenberg, which sends $m = kN + m^\sharp_N$ to $\frac{p^n-1}{N}(N-m^\sharp_N) + (k+1)$.
	\begin{lem}
		Suppose $0\le a<p^n$ with $p^n\equiv 1\bmod MN$. Then
		\begin{align*}
			\left(\frac{a}{N}\right)^\flat_{p^n} = \frac{a + (-a/p^n)^\flat_Np^n}{N}\equiv \frac{a+(-a)^\flat_N}{N} \equiv \iota_N(a)\bmod M.
		\end{align*}
	\end{lem}
	\begin{proof}
		The first equality already appears in Proposition 3.2 \textit{op.~cit.}, and the second is formal. To prove the third, note that writing $a = kN + a^\sharp_N$, then $\iota_N(a) = (k+1) + \frac{p^n-1}{N}(N-a^\sharp_N)\equiv k+1\bmod M$. On the other hand, we do have $a+(-a)^\flat_N = (k+1)N$.
	\end{proof}}
	\begin{lem}
		Suppose $n\in \Z_{\ge 0}$ is such that $p^n\equiv 1\bmod MN$. We have
		\begin{align*}
			\mu_{\chi,N}(a+p^n\Z_p) = \sum_{0\le d<a} \chi(d)\got{h}_\emptyset(d) +(1-\chi(N)N)L(0,\chi),
		\end{align*}
		where $\got{h}_\emptyset(a) = -1 + N\1_{N\mid a}$.
	\end{lem}
	\begin{proof}
		We see that
		\begin{align*}
			\sum_{0\le d<a} \chi(d)\got{h}_\emptyset(d) = -\sum_{0\le d<a} \chi(a) + N\sum_{0\le d<a/N}\chi(Nd).
		\end{align*}
		Now, it is readily checked that $d<a/N$ if and only if $d<\frac{a+(-a)^\flat_N}{N}$. By Proposition 3.2 \textit{op.~cit.}, we have
		\begin{align*}
			\left(\frac{a}{N}\right)^\flat_{p^n} = \frac{a+(-a)^\flat_Np^n}{N}\equiv \frac{a+(-a)^\flat_N}{N} \bmod M.
		\end{align*}
		The result then follows from \eqref{eq:1} and the vanishing $\sum_{d\in \Z/M}\chi(d) = 0$.
	\end{proof}

	\begin{cor}
		Notation as above. Let $q>1$ be a power of $p$ such that $q\equiv 1\bmod MN$. Then
		\begin{align*}
			L_{p,N}(s,\chi\omega)=(1-\chx{N}^{1-s}\chi\omega(N))L_p(s,\chi\omega) = \lim_{n\to\infty}\sum_{1\le m<q^n, p\nmid m}\chx{m}^{-s} \sum_{1\le d<m}\chi(d)\got{h}_{\emptyset}(d).
		\end{align*}
	\end{cor}
	
	\subsubsection{A Ferrero--Greenberg type formula}
	
	Assume now $\chi(p)=1$, so $L_p(0,\chi\omega) = L_{p,N}(0,\chi\omega)=0$, and $L_{p,N}'(0,\chi\omega) = (1-N\chi(N))L_p'(0,\chi\omega)$. We are ready to compute the derivative of $L_{p,N}(s,\chi\omega)$; note that in the limits below, $q$ denotes a power of $p$ with $q\equiv 1\bmod MN$.
	\begin{align}\label{eq:FG@N}
		\begin{split}
			L_{p,N}'(0,\chi\omega) &= -\lim_{n\to \infty}\sum_{\substack{0\le m<q^n\\ p\nmid m}}\sum_{0\le d<m}\chi(d)\got{h}_\emptyset(d) \log_p m\\
			&=-\lim_{n\to \infty} \sum_{0\le d<MN}\chi(d)\got{h}_\emptyset(d)\sum_{\substack{0\le m<q^n\\
					p\nmid m, m^\sharp_{MN}>d}} \log_p m\\
			&=- \lim_{n\to \infty} \sum_{0\le d<MN}\chi(d)\got{h}_\emptyset(d)
			\left[\sum_{\substack{0\le m<q^n\\ 					p\nmid m, m^\sharp_{MN}>d}}\log_p(MN) + \sum_{\substack{0\le m<\frac{MN-d}{MN}q^n\\ p\nmid m}}\log_p m\right]\\
			&=-\log_p (MN) L_{p,N}(0,\chi\omega) - \sum_{0\le d<MN}\chi(d)\got{h}_\emptyset(d)\log_p\Gamma_p\left(\frac{d}{MN}\right)\\
			&=\sum_{0\le d<M} \chi(d) \log_p \left(\frac{\prod_{0\le k<N}\Gamma_p\left(\frac{d+kM}{MN}\right)}{\Gamma_p\left(\frac{d+(-d/M)^\flat_N M}{MN}\right)^N}\right).
		\end{split}
	\end{align}
	Here in the third equality we employed the Ferrero--Greenberg permutation from Appendix A, \textit{op.~cit.}, which is a set bijection
	\begin{align*}
		\iota: \Phi_d =\{1\le m<p^n:p\nmid m, m^\sharp_{MN}>d\}\longrightarrow \Psi_d =\left\{1\le m<\frac{MN-d}{MN}p^n:p\nmid m\right\},
	\end{align*}
	and is such that $m\equiv MN\iota(m)\bmod p^n$.
	
	Next, by the original Ferrero--Greenberg formula \cite{FG78} and the relation $L_{p,N}'(0,\chi\omega) = (1-\chi(N)N)L_p'(0,\chi)$, we find
	\begin{align*}
		\sum_{0\le d<M} \chi(d) \log_p \left(\frac{\prod_{0\le k<N}\Gamma_p\left(\frac{d+kM}{MN}\right)}{\Gamma_p\left(\frac{d+(-d/M)^\flat_N M}{MN}\right)^N}\right) = 
		\sum_{0\le d<M} \chi(d) \log_p\left(\frac{\Gamma_p\left(\frac{d}{M}\right)}{\Gamma_p\left(\frac{(d/N)^\flat_M}{M}\right)^N}\right).
	\end{align*}
	As $\frac{d+(-d/M)^\flat_N M}{N} = (d/N)^\flat_M$, we have a simplified formula
	\begin{align}\label{eq:comparison.chi}
		\sum_{0\le d<M} \chi(d) \log_p \prod_{0\le k<N}\Gamma_p\left(\frac{d+kM}{MN}\right) = 
		\sum_{0\le d<M} \chi(d) \log_p\Gamma_p\left(\frac{d}{M}\right).
	\end{align}
	We remark that \eqref{eq:comparison.chi} remains valid if $\chi$ is the trivial character on $(\Z/M)^\times$, because $\Gamma_p(z)\Gamma_p(1-z) = \pm 1$. 
	
	Now, let $\nu$ be the order of $p$ in the multiplicative group $(\Z/M)^\times$. Varying $\chi$ in the group of characters on $(\Z/M)^\times/(p)$, for $0\le d<M$ with $\gcd(d,M)=1$, we find the following product formula
	\begin{align}\label{eq:gamma-distribution}
		\prod_{0\le j<\nu} \prod_{0\le k<N}\left\langle\Gamma_p\left(\frac{(p^jd)^\flat_M+kM}{MN}\right)\right\rangle = 
		\prod_{0\le j<\nu}\left\langle\Gamma_p\left(\frac{(p^jd)^\flat_M}{M}\right)\right\rangle.
	\end{align}

	\subsubsection{Arithmetic interpretation}
	We shall now assume that $p$ is an odd prime. Fix a prime $\got{P}|p$ in $\Q(\mu_M)$. Denote by $\F$ the residue field of $\got{P}$, which has cardinality $q = p^\nu$. Fix a multiplicative lift $t: \F^\times \to \mu_{q-1}\subset \bar{\Q}$ and an additive isomorphism $\psi: \F_p\xrightarrow{\sim} \mu_p\subset\bar{\Q}$. Recall for $h\in \frac{1}{M}\Z/\Z - \{0\}$, the finite field Gauss sum \cite{We49} is defined by (note the minus sign)
	\begin{align*}
		\got{g}(h) = -\sum_{x\in \F^\times} t(x^{(q-1)h})\psi(\tr_{\F/\F_p} x).
	\end{align*}
	Additionally, we will consider Gauss sums for $\Q(\mu_{MN})$. Thus take a prime $\got{P}'|\got{P}$ of $\Q(\mu_{MN})$, and denote its residue field by $\F'$; write $q' = p^{\nu'} = \#\F'$. Fix a multiplicative lift $t':\F'^\times\to \mu_{q'-1}$ such that $t'|_{\F^\times} = t$. Then for $h\in \frac{1}{MN}\Z-\{0\}$, we write
	\begin{align*}
		\got{g}'(h) = - \sum_{x\in \F'^\times} t(x^{(q'-1)h})\psi(\tr_{\F'/\F_p}x).
	\end{align*}
	Note that if $h\in \frac{1}{N}\Z-\{0\}$, then $t(x^{(q'-1)h}) = t(\nm_{\F'/\F}(x)^{(q-1)h})$, so in this case we have the following identity due to Davenport--Hasse (see, e.g., p.~503 \textit{op.~cit.})
	\begin{align}\label{eq:DH-extension}
		\got{g}'(h) = \got{g}(h)^{\nu'/\nu}.
	\end{align} 
	
	We now relate the Morita Gamma values to the Gauss sums. First, we have
	\begin{align}\label{eq:nu'-nu}
		\prod_{0\le j<\nu} \prod_{0\le k<N}\Gamma_p\left(\frac{(p^jd)^\flat_M+kM}{MN}\right)^{\nu'/\nu} =
		\prod_{0\le j<\nu'} \prod_{0\le k<N}\Gamma_p\left(\frac{(p^jd+p^jkM)^\flat_{MN}}{MN}\right).
	\end{align}
	Fix a primitive $p$-th root of unity $\zeta$, and put $\pi = (\zeta-1)\sqrt[p-1]{\frac{-p}{(\zeta-1)^{p-1}}}$. By the Gross--Koblitz formula \cite[Corollary 1.11]{GK79}, we find
	\begin{align}\label{eq:GK1}
		\prod_{0\le j<\nu'} \prod_{0\le k<N}\Gamma_p\left(\frac{(p^jd+p^jkM)^\flat_{MN}}{MN}\right) = 
		\pi^{(\nu'/\nu)\sum_{0\le j<\nu}(p^j d)^\flat_{M}/M + \nu'(N-1)/2}
		\prod_{0\le k<N}\got{g}'\left(\frac{d+kM}{MN}\right)
	\end{align}
	and
	\begin{align}\label{eq:GK2}
		\prod_{0\le j<\nu}\Gamma_p\left(\frac{(p^jd)^\flat_M}{M}\right) = (-p)^{\sum_{0\le j<\nu}(p^j d)^\flat_M/M}
		\got{g}\left(d/M\right).
	\end{align}
	Combining \eqref{eq:gamma-distribution}, \eqref{eq:DH-extension}, \eqref{eq:nu'-nu}, \eqref{eq:GK1} and \eqref{eq:GK2}, we deduce that
	\begin{align*}
		\sum_{0\le k<N} \log_p\got{g}'\left(\frac{d+kM}{MN}\right) =
		\log_p\got{g}'\left(\frac{d}{M}\right).
	\end{align*}
	We have thus recovered a weak form of the Davenport--Hasse relation \cite[Theorem 2.10.1]{Lang}.

	\bibliographystyle{alpha}
	\bibliography{gamma.bib}
\end{document}